\documentclass[11pt]{article}
\usepackage{amsmath, amsfonts, amssymb, amsthm}
\newtheorem{Theorem}{Theorem}[section]

\newtheorem{Lemma}[Theorem]{Lemma}

\newtheorem{Corollary}[Theorem]{Corollary}

\newtheorem{Open Problem}[Theorem]{Open Problem}

\newcommand{\mysection}[1]{\section{#1}\setcounter{equation}{0}}

\begin{document}

\title{Multiple nonsymmetric nodal solutions for quasilinear Schr\"{o}dinger system }
\author{Jianqing Chen$^{1,2}$\ and Qian Zhang$^{3,}$\thanks{Corresponding author. Email:\ zhangqianmath@163.com}\\
\small  \  1.\ School of Mathematics and Statistics, Fujian Normal University, \\
\small  Qishan Campus, Fuzhou 350117, P. R. China\\
\small  \  2.\ FJKLMAA and Center for Applied Mathematics of Fujian Province(FJNU), \\
\small  Fuzhou, 350117, P. R. China\\
\small  \  3.\  Department of Mathematical Sciences, Tsinghua University, Beijing 100084, China\\
\small  jqchen@fjnu.edu.cn (J. Chen) \ \   zhangqianmath@163.com  (Q. Zhang) }

\date{}

\maketitle
\noindent {\bf Abstract}: In this paper, we consider the quasilinear Schr\"{o}dinger system in $\mathbb R^{N}$($N\geq3$):
$$\left\{\aligned &-\Delta u+ A(x)u-\frac{1}{2}\triangle(u^{2})u=\frac{2\alpha }{\alpha+\beta}|u|^{\alpha-2}u|v|^{\beta},\\
&-\Delta v+ Bv-\frac{1}{2}\triangle(v^{2})v=\frac{2\beta }{\alpha+\beta}|u|^{\alpha}|v|^{\beta-2}v,\endaligned\right. $$
where $\alpha,\beta>1$, 
$2<\alpha+\beta<\frac{4N}{N-2}$, 
$B>0$ is a constant. By using a constrained minimization on Nehari-Poho\v{z}aev set, for any given integer $s\geq2$, we construct a non-radially symmetrical nodal solution with its $2s$ nodal domains.

\medskip



\mysection {Introduction}

We study the following  quasilinear Schr\"{o}dinger system
\begin{equation}\label{eq1.1}
\left\{\aligned &-\Delta u+A(x)u-\frac{1}{2}\triangle(u^{2})u=\frac{2\alpha }{\alpha+\beta}|u|^{\alpha-2}u|v|^{\beta},\\
&-\Delta v+B v-\frac{1}{2}\triangle(v^{2})v=\frac{2\beta }{\alpha+\beta}|u|^{\alpha}|v|^{\beta-2}v,\endaligned\right.
\end{equation}
where $u(x)\to 0,\ \ v(x)\to 0$ as $|x|\to\infty$, $N\geq 3$, $u:=u(x),v:=v(x)$ 
be real valued functions on $\mathbb R^{N}$, $\alpha,\ \beta>1$, $2<\alpha+\beta<\frac{4N}{N-2}$,  $B>0$ is a constant. In the last two decades, much attention has been devoted to the quasilinear Schr\"{o}dinger equation of the form
\begin{equation}\label{eq12}
-\Delta u+V(x)u-\frac{1}{2}u\Delta(u^{2})=|u|^{p-2}u,\ x\in \mathbb R^{N}.
\end{equation}
The equation (\ref{eq12}) is related  to the existence of standing waves of the following quasilinear Schr\"{o}dinger equation
\begin{equation}\label{eq13}
i\partial_{t}z=-\Delta z+V(x)z-l(|z|^{2})z-\frac{1}{2}\Delta g(|z|^{2})g'(|z|^{2})z,\ x\in \mathbb R^{N},
\end{equation}
where $V$ is a given potential, $l$ and $g$ are real functions. The equation (\ref{eq13}) has been used as models in several areas of physics corresponding to various types of $g$. The superfluid film equation in plasma physics has this structure for $g(s)=s$ \cite{k}. In the case $g(s)=(1+s)^{\frac{1}{2}}$, the equation (\ref{eq13}) models the self-channeling of a high-power ultra short laser in matter \cite{r}. The equation (\ref{eq13}) also appears in fluid mechanics \cite{k,lss}, in the theory of Heidelberg ferromagnetism and magnus \cite{ltz}, in dissipative quantum mechanics and in condensed matter theory \cite{mf}.  When considering the case $g(s)= s$, one obtains a corresponding equation of elliptic type like (\ref{eq12}). For more detailed  mathematical and physical interpretation of equations like (\ref{eq12}), we refer to \cite{bmmlb,bl,cs,lw,psw,0s} and the references therein.

In recent years, there has been increasing interest in studying problem (\ref{eq12}), see for examples, \cite{ctc,cl,jlw,m,ma,wz,w} and the references therein. More precisely, by Mountain-Pass theorem and the principle of symmetric criticality, Severo \cite{s} obtained symmetric and nonsymmetric solutions for  quasilinear Schr\"{o}dinger equation (\ref{eq12}). In \cite{lww}, when $4\leq p<\frac{4N}{N-2}$, Liu, Wang and Wang established  the existence results of a positive ground
state solution and a sign-changing ground state solution were given by using the Nehari method for (\ref{eq12}). Based on the method of perturbation and invariant sets of descending flow, Zhang and Liu \cite{zl} studied the nonautonomous case of (\ref{eq12}), they obtained the existence of infinitely many sign-changing solutions for $4< p<\frac{4N}{N-2} $. With the help of Nehari method and change of variables, Deng, Peng and Wang \cite{dpw} considered
\begin{equation}\label{eq1.4}
-\Delta u+V(x)u- u\Delta(u^{2})=\lambda|u|^{p-2}u+|u|^{\frac{4N}{N-2}-2}u,\ x\in \mathbb R^{N},
\end{equation}
and proved that (\ref{eq1.4}) has at least one pair of $k$-node solutions if either $N\geq6$ and   $4< p<\frac{4N}{N-2} $ or $3\leq N < 6$ and $ \frac{2(N+2)}{N-2} <q< \frac{4N}{N-2}$. In addition, problem (\ref{eq1.4}) still has at least one pair of $k$-node solutions if $3\leq N <6$ , $4 <q \leq \frac{2(N+2)}{N-2}$ and $\lambda $ sufﬁciently large. Note that all sign-changing solutions obtained in \cite{dpw,lww,zl}  are only valid for $4< p<\frac{4N}{N-2}$. When $2<p<\frac{4N}{N-2}$, Ruiz and Siciliano \cite{rs} showed equation (\ref{eq12}) has a ground states solution via Nehari-Poho\v{z}aev type constraint and concentration-compactness lemma, Wu and Wu \cite{ww} obtained the existence of radial solutions for (\ref{eq12}) by using change of variables.

It is natural to pose a series of interesting questions: whether we 
can find an unified approach to obtain sign-changing solutions 
for the full subcritical range of $2<\alpha+\beta<\frac{4N}{N-2}$?
Further, whether we can extend these results to  quasilinear 
Schr\"{o}dinger system? To answer these two questions, we adopt an action of finite subgroup $G$ of $O(2)$ from Szulkin and Waliullah \cite{sw} and look for the existence of non-radially symmetrical nodal solutions for quasilinear Schr\"{o}dinger system (\ref{eq1.1}).

Before stating our main results, we make the following assumptions:

\noindent$ (A_{1})$ $A\in C^{1}(\mathbb R^{N},\mathbb R^{+}), 0<A_{0}\leq A(x)\leq A_{\infty}=\lim\limits_{|x|\rightarrow\infty}A(x)<+\infty;$\\
$(A_{2})$ $  \nabla A(x)\cdot x \in L^{\infty}(\mathbb R^{N})$, $(\alpha+\beta-2)A(x)-\nabla A(x)\cdot x\geq0$; \\
$(A_{3})$ the map $s\mapsto s^{\frac{N+2}{N+\alpha+\beta}}A(s^{\frac{1}{N+\alpha+\beta}}x)$ is concave for any $x\in \mathbb R^{N}$;\\
$(A_{4})$  $A(x)$ is radially symmetric with respect to the first two coordinates, that is to say, if
$(x_{1},x_{2},x_{3},\cdots,x_{N})$, $(y_{1},y_{2},y_{3},\cdots,y_{N} )\in\mathbb R^{N}$ and  $x^2_1 + x^2_2  =  y^2_1 + y^2_2$,  then $A(x_{1},x_{2},z_{3},\cdots,z_{N} ) = A(y_{1},y_{2},z_{3},\cdots,z_{N})$.

It is worth noting that  $(A_{1})$ is used to derive the existence of a 
strongly convergent subsequence, while for the system, we only need one 
such kind of condition in our system, which seems to be a different phenomenon due to the coupling of $u$ and $v$. $(A_{2})$-$(A_{3})$ once appeared in \cite{ rs,ww} to obtain the existence of ground states solutions for the quasilinear Schr\"{o}dinger equation. $(A_{4})$ once appeared in \cite{cyh} to prove the existence of nodal solutions of single $p$-Laplacian equation.

 Our main results  reads as follows.
\begin{Theorem}\label{th1.1}
Assume that $(A_{1})$-$(A_{4})$ hold. For any given integer $s\geq2$, the problem (\ref{eq1.1}) possesses  a non-radially symmetrical nodal solution with its $2s$ nodal domains.
\end{Theorem}

\begin{Corollary}\label{co12}
If $A(x)$ is a positive constant, one can still obtain the same results as  Theorem \ref{th1.1} for system (\ref{eq1.1}).
\end{Corollary}

\noindent {\bf Remark 1.3.}  Since $s\in \mathbb N$ is arbitrary, the solution we obtained in Theorem \ref{th1.1} is actually a result of multiplicity.
\vskip4pt
\noindent {\bf Remark 1.4.}  As a main novelty with respect to some results 
in \cite{dpw,lww,zl}, we are able to deal with exponents 
$\alpha+\beta\in(2,\frac{4N}{N-2})$ and obtain the existence and multiplicity 
of nodal solution.
\vskip4pt

The rest of the paper is organized as follows. In Section 2, we establish some preliminary results. Theorem \ref{th1.1} is proved in Section 3.

\mysection {Preliminaries}

Throughout this paper, $\|u\|_{H^1}$ and $|u|_r$ denote the usual 
norms of $H^1(\mathbb R^N)$ and $L^r(\mathbb R^N)$ for $r>1$, 
respectively. $C$ and $C_{i}(i=1,2,\ldots)$ denote 
(possibly different) positive constants and $\int_{\mathbb R^{N}}g$ denotes the integral $\int_{\mathbb R^{N}}g(z)dz$. The $\rightarrow$ and $\rightharpoonup$ denote strong convergence and weak convergence, respectively.

Let $H^{1}(\mathbb R^{N})$ be the usual Sobolev space, define $X :=\ H\times H$
with  $H := \{ u \in H^{1}(\mathbb R^{N})\   | \ \int_{\mathbb R^{N}}u^{2}|\nabla u|^{2}<+\infty  \}.$
The term $\int_{\mathbb R^{N}}u^{2}|\nabla u|^{2}$ is not convex
 and $H$ is not even a vector space. So, the usual min-max 
 techniques cannot be directly applied, nevertheless $H$ is a
  complete metric space with distance
$$ d_{H}(u,\omega)=\|u-\omega\|_{H^{1}}+|\nabla u^{2}-\nabla \omega^{2}|_{2}.  $$
Define
$$\aligned d_{X}\bigl((u,v),(\omega,\nu)\bigr)
&:=\|u-\omega\|_{H^{1}}+|\nabla u^{2}-\nabla \omega^{2}|_{2}+\|v-\nu\|_{H^{1}}\\
&\ \ \ +|\nabla v^{2}-\nabla \nu^{2}|_{2}. \endaligned$$
Then we call $(u,v)\in X$ is a weak solution of (\ref{eq1.1}) if for any $\varphi_{1},\varphi_{2}\in C_{0}^{\infty}(\mathbb R^{N})$,
$$\aligned &\int_{\mathbb R^{N}}\bigg((1+u^{2})\nabla u\nabla \varphi_{1}+\left(u|\nabla u|^{2}+A(x)u-\frac{2\alpha|u|^{\alpha-2}u|v|^{\beta}}{\alpha+\beta}\right)\varphi_{1}\bigg)=0,\endaligned$$
and
$$\aligned &\int_{\mathbb R^{N}}\bigg((1+v^{2})\nabla v\nabla \varphi_{2}+\left(v|\nabla v|^{2}+B v-\frac{2\beta|u|^{\alpha}|v|^{\beta-2}v}{\alpha+\beta}\right)\varphi_{2}\bigg)=0.\endaligned$$
Hence there is a one-to-one correspondence between solutions of (\ref{eq1.1}) and critical points of the following  functional $I: X \rightarrow\mathbb R$ defined by
\begin{equation}\label{eq21}
\aligned I(u,v)&=\frac{1}{2}\int_{\mathbb R^{N}}(|\nabla u|^{2}+|\nabla v|^{2}+A(x)u^{2}+B v^{2})\\
&\ \ \ \ +\frac{1}{2}\int_{\mathbb R^{N}} (u^{2}|\nabla u|^{2}+v^{2}|\nabla v|^{2})-\frac{2}{\alpha+\beta}\int_{\mathbb R^{N}} |u|^{\alpha}|v|^{\beta} .\endaligned
\end{equation}
For any $\varphi_{1},\varphi_{2}\in C_{0}^{\infty}(\mathbb R^{N})$, $(u,v)\in X,$ and $(u,v)+(\varphi_{1},\varphi_{2})\in X$, we compute the Gateaux derivative
$$\aligned &\ \langle I'(u,v),(\varphi_{1},\varphi_{2})\rangle\\
=& \int_{\mathbb R^{N}}\bigl((1+u^{2})\nabla u\nabla \varphi_{1}+(1+v^{2})\nabla v\nabla \varphi_{2}+u|\nabla u|^{2}\varphi_{1}\\
&\ +v|\nabla v|^{2}\varphi_{2}+A(x)u\varphi_{1}+B v\varphi_{2}\bigr)-\frac{2\alpha}{\alpha+\beta}\int_{\mathbb R^{N}}|u|^{\alpha-2}u|v|^{\beta}\varphi_{1}\\
&\  -\frac{2\beta}{\alpha+\beta}\int_{\mathbb R^{N}}|v|^{\beta-2}v|u|^{\alpha}\varphi_{2}.\endaligned$$
Then, $(u,v)\in X $ is a solution of (\ref{eq1.1}) if and only if
$$\langle I'(u,v),(\varphi_{1},\varphi_{2})\rangle=0,\ \  \varphi_{1},\ \varphi_{2}\in C_{0}^{\infty}(\mathbb R^{N}) .$$

Motivated by \cite{sw}, we recall that a subset $U$ of a Banach space $\mathbb{E}$ is called invariant with respect to an action of a group $G$ (or $G$-invariant) if $gU\subset U$ for all $g\in G$, and a functional $I : U\rightarrow \mathbb R$ is invariant (or $G$-invariant) if $I(gu) = I(u)$ for all $g \in G$, $u\in U$. The subspace
$$\mathbb{E}_{G}:=\{u\in \mathbb{E}\ |\  gu=u\ \ \hbox{for all}\  g\in G\}$$ is called the fixed point space of this action.

Let $x =(y,z) = (y_{ 1} ,y_{ 2} ,z_{ 1} ,\cdots,z_{ N} ) 
\in\mathbb  R^{ N}$ and let $O(2)$ be the group of orthogonal 
transformations acting on $\mathbb R^{ 2}$ by $(g,y)\mapsto gy.$ 
For any positive integer $s$ we define $G_{s}$ to be the finite 
subgroup of $O(2)$ generated by the two elements $\alpha$ and 
$\beta$ in $O(2)$, where $\alpha$ is the rotation in the $y$-plane 
by the angle $\frac{2\pi}{s}$ and $\beta$ is the reflection in 
the line $y_{ 1} = 0 $ if $s= 2$, and in the line $y_{ 2} 
= \hbox{tan}\bigl(\frac{\pi}{s}\bigr)y_{ 1}$ for other $s$ (so in complex notation $w = y_{ 1} + iy_{ 2}$ , $\alpha w = we^{ \frac{2\pi i}{s}}$ and $\beta w = we^{ \frac{2\pi i}{s} }$ ).

$\forall\ g\in G_{ s}$, $x\in\mathbb R^{N},\ gx:=(gy,z)$. Define the action of $G_{s}$ on $H^{1}(\mathbb R^{N})$ by setting
$$ (g(u,v))x:=(gu,gv)x=(\det(g)ug^{-1}x,\det(g)vg^{-1}x).$$
Define
$$\mathcal{V}:=\{(u,v)\in X\ |\ (u,v)(gx)=(\det(g)u(x),\det(g)v(x)),\  \ g\in G_{s}\}, $$
$$\mathcal{M}:=\{(u,v)\in \mathcal{V}\backslash\{(0,0)\}\ |\ \mathcal{G}(u,v)=0 \},$$
where $\mathcal{G}:X\rightarrow \mathbb R$ and
$$ \aligned \mathcal{G}(u,v)&=\frac{N}{2} \int_{\mathbb R^{N}}(|\nabla u|^{2}+|\nabla v|^{2})+\frac{N+2}{2} \int_{\mathbb R^{N}}(A(x)u^{2} +Bv^{2} \\
&\ \ \ \ +u^{2}|\nabla u|^{2}+v^{2}|\nabla v|^{2})-\frac{2(N+\alpha+\beta)}{\alpha+\beta}\int_{\mathbb R^{N}}|u|^{\alpha}|v|^{\beta}. \endaligned $$
Let
\begin{equation}\label{eq2.2}
m:=\inf_{(u,v)\in \mathcal{M}}I(u,v).
\end{equation}
Then our aim is to prove that $m$ is achieved. In the rest of this section, we will give some properties of the set $\mathcal{M}$.

For any $u\in H^1(\mathbb R^N)$, we define $u_{t}:\mathbb R^{+}\rightarrow H^1(\mathbb R^N)$ by:
$$ u_{t}(x):=tu(t^{-1}x).$$
Let $t\in \mathbb R^{+}$ and $(u,v)\in X$. We have that
$$\aligned I(u_{t},v_{t})&=\frac{t^{N}}{2}\int_{\mathbb R^{N}}(|\nabla u|^{2}+|\nabla v|^{2})+\frac{t^{N+2}}{2}\int_{\mathbb R^{N}}(A(x)u^{2}+Bv^{2}\\
&\ \ \ \ +u^{2}|\nabla u|^{2}+ v^{2}|\nabla v|^{2}) -\frac{2t^{N+\alpha+\beta}}{\alpha+\beta}\int_{\mathbb R^{N}}|u|^{\alpha}|v|^{\beta}.\endaligned $$
Denote $h_{uv}(t):=I(u_{t},v_{t})$. Since $\alpha+\beta>2,$ we see that $h_{uv}(t)>0$ for $t>0$ small enough and $h_{uv}(t)\rightarrow -\infty$ as $t\rightarrow +\infty,$ this implies that $h_{uv}(t)$ attains its maximum. Moreover,  $h_{uv}(t): \mathbb R^{+}\rightarrow \mathbb R$ is $C^{1}$ and
$$\aligned h'_{uv}(t)
&=\frac{N}{2}t^{N-1}\int_{\mathbb R^{N}}(|\nabla u|^{2}+|\nabla v|^{2} )+\frac{N+2}{2}t^{N+1}\int_{\mathbb R^{N}}(A(x)u^{2} +Bv^{2} \\
&\ \ \ \ +u^{2}|\nabla u|^{2}+ v^{2}|\nabla v|^{2}) -\frac{2(N+\alpha+\beta)}{\alpha+\beta}t^{N+\alpha+\beta-1}\int_{\mathbb R^{N}}|u|^{\alpha}|v|^{\beta}.\endaligned$$

\begin{Lemma}\label{le2.1}
If $(u,v)\in X$ is a weak solution of (\ref{eq1.1}), then $(u,v)$ satisfies the following $ P(u,v)=0$, where
\begin{equation}\label{eq2.1}
\aligned P(u,v):&=\frac{N-2}{2}\int_{\mathbb R^{N}}(|\nabla u|^{2}+|\nabla v|^{2}+ u^{2}|\nabla u|^{2}+ v^{2}|\nabla v|^{2}) \\
&\ \ \ +\frac{N}{2}\int_{\mathbb R^{N}}(A(x)u^{2}+Bv^{2})+\frac{1}{2}\int_{\mathbb R^{N}}\nabla A(x)\cdot xu^{2}\\
&\ \ \ -\frac{2N}{\alpha+\beta}\int_{\mathbb R^{N}}|u|^{\alpha}|v|^{\beta}.\endaligned
\end{equation}
\end{Lemma}
\begin{proof}
The proof is standard, so we omit it here.
\end{proof}

The lemma below shows (\ref{eq2.2}) is well defined.

\begin{Lemma}\label{le2.2} For any $(u,v)\in X$ and $u,v\neq0$, the map $h_{uv}$ attains its maximum at exactly one point $\bar{t}$. Moreover, $h_{uv}$ is positive and increasing for $t\in [0,\bar{t}]$ and decreasing for $t>\bar{t}$. Finally
$$m=\inf_{(u,v)\in X}\max_{t>0}I(u_{t},v_{t}). $$
\end{Lemma}
\begin{proof}
For any $t>0$, set $s=t^{N+\alpha+\beta}$, we obtain
$$\aligned h_{uv}(s)&=\frac{s^{\frac{N}{N+\alpha+\beta}}}{2}\int_{\mathbb R^{N}}|\nabla u|^{2}+\frac{s^{\frac{N}{N+\alpha+\beta}}}{2}\int_{\mathbb R^{N}}|\nabla v|^{2}+\frac{s^{\frac{N+2}{N+\alpha+\beta}}}{2}\int_{\mathbb R^{N}}u^{2}|\nabla u|^{2}\\
&\ \ \ \ +\frac{s^{\frac{N+2}{N+\alpha+\beta}}}{2}\int_{\mathbb R^{N}}v^{2}|\nabla v|^{2}+\frac{s^{\frac{N+2}{N+\alpha+\beta}}}{2}\int_{\mathbb R^{N}}A(s^{\frac{1}{N+\alpha+\beta}}x)u^{2}\\
&\ \ \ \ +\frac{s^{\frac{N+2}{N+\alpha+\beta}}}{2}\int_{\mathbb R^{N}}Bv^{2}-\frac{2s }{\alpha+\beta}\int_{\mathbb R^{N}}|u|^{\alpha}|v|^{\beta}.\endaligned $$
This is a concave function by condition $(A_{3})$ and we already know that it attains its maximum, let $\bar{t}$ be the unique point at which this maximum is achieved. Notice that $\mathcal{G}(u_{t},v_{t})=th'_{uv}(t)$, then $\bar{t}$ is the unique critical point of $h_{uv}$  and $h_{uv}$ is positive and increasing for $0<t <\bar{t}$ and decreasing for $t>\bar{t}$. In particular,  $\bar{t}\in \mathbb R$ is the unique value such that $u_{\bar{t}} \in\mathcal{M},$ and $I(u_{\bar{t}},v_{\bar{t}} )$ reaches a global maximum for $t=\bar{t}$. This finishes the proof.
\end{proof}

\begin{Lemma}\label{le2.3}
 $m>0$.
\end{Lemma}

\begin{proof}
For every $(u,v)\in \mathcal{M} $, it follows from $(A_{2})$ that
$$\aligned I(u,v)&=\frac{\alpha+\beta}{2(N+\alpha+\beta)}\int_{\mathbb R^{N}}(|\nabla u|^{2}+ |\nabla v|^{2})\\
&\ \ \ + \frac{\alpha+\beta-2}{2(N+\alpha+\beta)}\int_{\mathbb R^{N}}(Bv^{2}+u ^{2}|\nabla u |^{2}+ v ^{2}|\nabla v |^{2})\\
&\ \ \ +\frac{1}{2(N+\alpha+\beta)}\int_{\mathbb R^{N}}\bigl((\alpha+\beta-2)A(x)-\nabla A(x)\cdot x\bigr)u^{2}\\
&>0.\endaligned $$
The proof is complete.
\end{proof}

\mysection{Proof of Theorem \ref{th1.1} }

We need the following variant of Lions lemma.

\begin{Lemma}\label{le2.70}
If $q\in[2,\frac{4N}{N-2}),$ $\{u_{n}\}$ is bounded in $X$, $r_{ 0}>0$ is such that for all $r\geq r_{0}$
\begin{equation}\label{eq3.2} \lim_{n\rightarrow\infty}\sup_{z\in \mathbb R}\int_{B((0,z),r)}|u_{n}|^{q}=0,
\end{equation}
then we have $u_{n}\rightarrow0$ in $ L^{p}(\mathbb R^{N})$ for $p\in(2,\frac{4N}{N-2}).$
\end{Lemma}
\begin{proof}
By using \cite[Lemma 2.2]{wz}, it remains to prove that for some $r > 0$,
$$ \lim_{n\rightarrow\infty}\sup_{z\in \mathbb R^{N}}\int_{B(z,r)}|u_{n}|^{q}=0.$$
Suppose that
\begin{equation}\label{eq3.3}
\int_{B(z_{n},1)}|u_{n}|^{q}\geq c>0.
\end{equation}
Observe that in the family $\{B(gz_{ n },1)\}_{g\in O(2)}$, we find an increasing number of disjoint balls provided that $|(z_{n}^{ 1},z_{n}^{ 2})|\rightarrow\infty $. Since $\{u_{ n }\}$ is bounded in $L^{q} (\mathbb R^{ N }),q\in[2,\frac{4N}{N-2})$, by (\ref{eq3.3}), $|(z_{n}^{ 1},z_{n}^{ 2})|$ must be bounded. Then for sufficiently large $r\geq r_{0}$, one obtains
$$\int_{B((0,z_{n}^{3}),r)}|u_{n}|^{q}\geq\int_{B(z_{n},1)}|u_{n}|^{q}\geq c>0,$$
and we get a contradiction with (\ref{eq3.2}).
\end{proof}

\begin{Lemma}\label{le3.2}
Let  $u_{n}\rightharpoonup u, v_{n}\rightharpoonup v$  in $X$, $u_{n}\rightarrow u, v_{n}\rightarrow v$ a.e in $\mathbb R^{N}$. Then
$$\aligned
&\lim_{n\rightarrow\infty}\int_{\mathbb R^{N}}|u_{n}|^{\alpha}|v_{n} |^{\beta}-\int_{\mathbb R^{N}}|u|^{\alpha}|v|^{\beta} =\lim_{n\rightarrow\infty}\int_{\mathbb R^{N}}|u_{n}-u|^{\alpha}|v_{n}-v|^{\beta}.
\endaligned$$
\end{Lemma}

\begin{proof}
For $n=1,\ 2,\ldots$, we have that
$$\begin{aligned}
&\ \int_{\mathbb R^{N}}|u_n|^\alpha |v_n|^\beta -\int_{\mathbb R^{N}}|u_n-u|^\alpha |v_n-v|^\beta\\
=&\int_{\mathbb R^{N}}(|u_n|^\alpha-|u_n-u|^\alpha) |v_n|^\beta+\int_{\mathbb R^{N}}|u_n-u|^\alpha (|v_n|^\beta-|v_n-v|^\beta).
\end{aligned}$$
Since $u_n\rightharpoonup u, v_n\rightharpoonup v$ in $H^1 (\mathbb{R}^N)$, from \cite[Lemma 2.5]{mj}, one has
$$\int_{\mathbb R^{N}}(|u_n|^\alpha-|u_n-u|^\alpha-|u|^\alpha)^{\frac{p}{\alpha}}\rightarrow0,\ \ n\rightarrow\infty,$$
which means that $|u_n|^\alpha-|u_n-u|^\alpha \rightarrow |u|^\alpha $  in $L^{\frac{p}{\alpha}} (\mathbb{R}^N)$.
Using $|v_n|^\beta \rightharpoonup |v|^\beta$ in $L^{\frac{p}{\beta}} (\mathbb{R}^N)$, it follows from $\alpha +\beta = p$ that
$$\int_{\mathbb R^{N}}(|u_n|^\alpha-|u_n-u|^\alpha) |v_n|^\beta\rightarrow\int_{\mathbb R^{N}}|u|^\alpha |v|^\beta,\ \ n\rightarrow\infty.$$
Similarly, $|v_n|^\beta-|v_n-v|^\beta \rightarrow |v|^\beta $  in $L^{\frac{p}{\beta}} (\mathbb{R}^N)$.
As $|u_n-u|^\alpha \rightharpoonup 0$ in $L^{\frac{p}{\alpha}} (\mathbb{R}^N)$, we obtain that
$$\int_{\mathbb R^{N}}|u_n-u|^\alpha (|v_n|^\beta-|v_n-v|^\beta)\rightarrow0,\ \ n\rightarrow\infty.$$
This proves the lemma.
\end{proof}

The following Lemma is due to  Poppenberg, Schmitt and Wang  from \cite[Lemma 2]{psw}.
\begin{Lemma}\label{le33}
Assume that  $u_{n}\rightharpoonup u$ in $H^{1}(\mathbb R^{N})$. Then
\begin{equation}\label{eq312}
\aligned\liminf_{n\rightarrow\infty}\int_{\mathbb R^{N}} u_{n}^{2}|\nabla u_{n}|^{2} &\geq \liminf_{n\rightarrow\infty}\int_{\mathbb R^{N}}(u_{n}-u)^{2}|\nabla u_{n}-\nabla u|^{2}\\
&\ \ \ \ +\int_{\mathbb R^{N}} u^{2}|\nabla u|^{2}.\endaligned
\end{equation}
\end{Lemma}

\begin{proof}
The proof is analogous to that of \cite[Lemma 2]{psw}, so we omit it here.
\end{proof}
\begin{Lemma}\label{le3.3}  $m$ is achieved at some $(u,v)\in  \mathcal{M}$.
\end{Lemma}
\begin{proof}
Let $\{(u_{n},v_{n})\}\subset \mathcal{M}$ be a sequence such that $I(u_{n},v_{n})\rightarrow m$. Using $(u_{n},v_{n})\subset \mathcal{M}$ and $(A_{2})$, we may obtain
$$\aligned  1+m\geq&\ I(u_{n}  ,v_{n} )\\
=&\ \frac{\alpha+\beta}{2(N+\alpha+\beta
)}\int_{\mathbb R^{N}}(|\nabla u_{n}|^{2}+ |\nabla v_{n}|^{2})\\
&\ + \frac{\alpha+\beta-2}{2(N+\alpha+\beta)}\int_{\mathbb R^{N}}(Bv_{n}^{2}+u_{n}^{2}|\nabla u_{n}|^{2}+ v_{n}^{2}|\nabla v_{n}|^{2})\\
&\  +\frac{1}{2(N+\alpha+\beta)}\int_{\mathbb R^{N}}\bigl((\alpha+\beta-2)A(x)-\nabla A(x)\cdot x\bigr)u_{n}^{2},\endaligned$$
which implies that $\{u _{n }\}$, $\{v_{ n} \}$, $\{u ^{2}_{n} \}$ and 
$\{v^{2}_{n} \}$ are bounded in $H^{1}( \mathbb R^{ N} )$, then, 
there exists a subsequence of  $(u_{n},v_{n})$, still denoted by 
 $(u_{n},v_{n})$ such that $(u_{n},v_{n})\rightharpoonup (u,v)$ in $X$. 
 Then $\{u_{n}\}$ and $\{v_{n}\}$ are 
 bounded in $L^{\alpha+\beta}(\mathbb R^{N})$. 
 The proof includes the following three steps.

{\bf Step 1.} $\int_{\mathbb R^{N}}|u_{n}|^{\alpha}|v_{n}|^{\beta}\not\rightarrow0$. It follows from Lemma \ref{le2.3} that
$$ \aligned I(u_{n},v_{n})&=\int_{\mathbb R^{N}}\bigg(\frac{1}{2}(|\nabla u_{n}|^{2}+|\nabla v_{n}|^{2}+A(x)u_{n}^{2}+Bv_{n}^{2}\\
&\ \ \ \ +u_{n}^{2}|\nabla u_{n}|^{2}+v_{n}^{2}|\nabla v_{n}|^{2})-\frac{2}{\alpha+\beta} |u_{n}|^{\alpha}|v_{n}|^{\beta}\bigg)\\
&\rightarrow m>0,\endaligned$$
then $\int_{\mathbb R^{N}}(|\nabla u_{n}|^{2}+|\nabla v_{n}|^{2}+A(x)u_{n}^{2}+Bv_{n}^{2}+u_{n}^{2}|\nabla u_{n}|^{2}+v_{n}^{2}|\nabla v_{n}|^{2})\not\rightarrow0.$
 By Lemma \ref{le2.2}, for $t > 1$,
$$\aligned m&\leftarrow I(u_{n},v_{n})\\
&\geq I\bigl((u_{n})_{t},(v_{n})_{t}\bigr)\\
&=\frac{t^{N}}{2}\int_{\mathbb R^{N}}(|\nabla u_{n}|^{2}+|\nabla v_{n}|^{2})+\frac{t^{N+2}}{2}\int_{\mathbb R^{N}}(A(tx)u_{n}^{2}+Bv_{n}^{2})\\
&\ \ \ +\frac{t^{N+2}}{2}\int_{\mathbb R^{N}}(u_{n}^{2}|\nabla u_{n}|^{2}+v_{n}^{2}|\nabla v_{n}|^{2}) -\frac{2t^{N+\alpha+\beta}}{\alpha+\beta}\int_{\mathbb R^{N}}|u_{n}|^{\alpha}|v_{n}|^{\beta}\\
&\geq \frac{t^{N}}{2}\int_{\mathbb R^{N}}(|\nabla u_{n}|^{2}+|\nabla v_{n}|^{2}+A_{0}u_{n}^{2} +Bv_{n}^{2}+u_{n}^{2}|\nabla u_{n}|^{2}+v_{n}^{2}|\nabla v_{n}|^{2})\\
&\ \ \ -\frac{2t^{N+\alpha+\beta}}{\alpha+\beta}\int_{\mathbb R^{N}}|u_{n}|^{\alpha}|v_{n}|^{\beta}\\
&\geq \frac{t^{N}}{2}\delta-\frac{2t^{N+\alpha+\beta}}{\alpha+\beta}\int_{\mathbb R^{N}}|u_{n}|^{\alpha}|v_{n}|^{\beta},\endaligned$$
where $\delta$ is a fixed constant. It suffices to choose $t > 1$ so that $ \frac{t^{N}\delta}{2}>2m$ to get a lower bound for $\int_{\mathbb R^{N}}|u_{n}|^{\alpha}|v_{n}|^{\beta}$.

Therefore, we may assume (passing to a subsequence, if necessary)  that
\begin{equation}\label{eq31}
\int_{\mathbb R^{N}}|u_{n}|^{\alpha}|v_{n}|^{\beta}\rightarrow D\in(0,\infty).
\end{equation}

{\bf Step 2.} $u\neq0$. By using (\ref{eq31}) and H\"{o}lder inequality, we can assume (passing to a subsequence, if necessary) that
$$
\int_{\mathbb R^{N}}|u_{n}|^{\alpha+\beta}>\delta>0.$$
By Lemma \ref{le2.70}, there exist $\delta > 0$ and $\{z_{n}\}\subset \mathbb R$ such that
\begin{equation}\label{eq032}
\limsup_{n\rightarrow+\infty}\int_{B((0,z_{n}),r)}|u_{n}|^{\alpha+\beta}>\delta>0.
\end{equation}
Define
$$y=(x_{1},x_{2}),\ \ z=(x_{3},\cdots,x_{N}),$$
and
$$w_{n}(x)=w_{n}(y,z)=u_{n}(y,z+z_{n}),\ \sigma_{n}(x)=\sigma_{n}(y,z)=v_{n}(y,z+z_{n}),$$ then $w_{n}\rightharpoonup w,\sigma_{n}\rightharpoonup \sigma$ in $X$.  In this case, by $(A_{4})$, we may obtain $I(u_{n},v_{n})=I(w_{n},\sigma_{n})$. By using (\ref{eq032}) and $w_{n}\rightarrow w$ in $L_{loc}^{\alpha+\beta}(\mathbb R^{N})$, one has
$$\aligned 0<\delta&<\limsup_{n\rightarrow+\infty}\int_{B((0,z_{n}),r)}|u_{n}|^{\alpha+\beta}\\
&=\limsup_{n\rightarrow+\infty}\int_{B((0,0),r)}|w_{n}|^{\alpha+\beta}\\
&=\int_{B((0,0),r)}|w|^{\alpha+\beta},\endaligned$$
which implies $w \neq 0$, and then $u\neq 0$.

{\bf Step 2.} We claim that $(u,v)\in \mathcal{M}.$ Indeed, if $(u,v)\not\in \mathcal{M},$ we discuss three cases:

{\bf Case 1:} $\mathcal{G}(u,v)<0$. By Lemma \ref{le2.2}, there exists $t\in(0,1)$ such that $(u _{t},v_{t})\in \mathcal{M},$ it follows from $(A_{2})$, $(u_{n},v_{n})\in \mathcal{M}$ and Fatou's lemma that
$$\aligned m&=\liminf_{n\rightarrow+\infty}\bigg(I( u_{n} ,v_{n} )-\frac{1}{N+\alpha+\beta}\mathcal{G}(u_{n} ,v_{n})\bigg)\\
&=\liminf_{n\rightarrow+\infty} \bigg(\frac{\alpha+\beta}{2(N+\alpha+\beta)}\int_{\mathbb R^{N}}(|\nabla u_{n}|^{2}+ |\nabla v_{n}|^{2})\\
&\ \ \ \ +\frac{1}{2(N+\alpha+\beta)}\int_{\mathbb R^{N}}\bigl((\alpha+\beta-2)A(x)-\nabla A(x)\cdot x\bigr)u_{n}^{2}\\
&\ \ \ \ +\frac{\alpha+\beta-2}{2(N+\alpha+\beta)}\int_{\mathbb R^{N}}(Bv_{n}^{2}+u_{n}^{2}|\nabla u_{n}|^{2}+v_{n}^{2}|\nabla v_{n}|^{2})\bigg)\\
&\geq \frac{\alpha+\beta}{2(N+\alpha+\beta)}\int_{\mathbb R^{N}}(|\nabla u|^{2}+ |\nabla v|^{2})\\
&\ \ \ \ +\frac{1}{2(N+\alpha+\beta)}\int_{\mathbb R^{N}}\bigl((\alpha+\beta-2)A(x)-\nabla A(x)\cdot x\bigr)u^{2}\\
&\ \ \ \ +\frac{\alpha+\beta-2}{2(N+\alpha+\beta)}\int_{\mathbb R^{N}}(Bv^{2}+u^{2}|\nabla u|^{2}+v^{2}|\nabla v|^{2})\endaligned$$
$$\aligned &>\frac{\alpha+\beta}{2(N+\alpha+\beta)}t^{N}\int_{\mathbb R^{N}}(|\nabla u|^{2}+ |\nabla v|^{2})\\
&\ \ \ \ +\frac{t^{N+2}}{2(N+\alpha+\beta)}\int_{\mathbb R^{N}}\bigl((\alpha+\beta-2)A(x)-\nabla A(x)\cdot x\bigr)u^{2}\\
&\ \ \ \ +\frac{\alpha+\beta-2}{2(N+\alpha+\beta)}t^{N+\alpha+\beta}\int_{\mathbb R^{N}}(Bv^{2}+u^{2}|\nabla u|^{2}+v^{2}|\nabla v|^{2})\\
&=I(u_{t},v_{t})-\frac{1}{N+\alpha+\beta}\mathcal{G}(u_{t} ,v_{t})\\
&\geq m, \endaligned$$
which is a contradiction.
\vskip4pt
{\bf Case 2:} $\mathcal{G}(u,v)>0$. Set $\xi_{n}:=u_{n}-u,\gamma_{n}:=v_{n}-v$, by Lemma \ref{le3.2}, Br\'{e}zis-Lieb lemma \cite{bli}, (\ref{eq312}), $(A_{1})$ and $(B_{1})$, we may obtain
\begin{equation}\label{eq3.7}
\mathcal{G}(u_{n},v_{n})\geq \mathcal{G}(u,v)+\mathcal{G}(\xi_{n},\gamma_{n})+o_{n}(1).
\end{equation}
Then $\limsup\limits_{n\rightarrow\infty}\mathcal{G}(\xi_{n},\gamma_{n})<0$.  By Lemma \ref{le2.2}, there exists $t_{n}\in(0,1)$ such that $((\xi_{n})_{t_{n}}, (\gamma_{n})_{t_{n}})\in \mathcal{M}.$ Furthermore, one has that $\limsup\limits_{n\rightarrow\infty}t_{n}<1$, otherwise, along a subsequence, $t_{n}\rightarrow1$ and hence
$$\mathcal{G}(\xi_{n},\gamma_{n})=\mathcal{G}((\xi_{n})_{t_{n}}, (\gamma_{n})_{t_{n}})+o_{n}(1) = o_{n}(1),$$
a contradiction. It follows from $(u_{n},v_{n})\in \mathcal{M} $,   (\ref{eq3.7}), $(A_{2}) $  that
$$\aligned &\ m+o_{n}(1)\\
=&\ I( u_{n} ,v_{n} )-\frac{1}{N+\alpha+\beta}\mathcal{G}(u_{n} ,v_{n})\\
=&\ \frac{\alpha+\beta}{2(N+\alpha+\beta)}\int_{\mathbb R^{N}}(|\nabla u_{n}|^{2}+|\nabla v_{n}|^{2})\\
&\ +\frac{1}{2(N+\alpha+\beta)}\int_{\mathbb R^{N}}\bigl((\alpha+\beta-2)A(x)-\nabla A(x)\cdot x\bigr)u_{n}^{2}\\
&\ +\frac{\alpha+\beta-2}{2(N+\alpha+\beta)}\int_{\mathbb R^{N}}(Bv_{n}^{2}+u_{n}^{2}|\nabla u_{n}|^{2}+v_{n}^{2}|\nabla v_{n}|^{2})\\
\geq&\ \frac{\alpha+\beta}{2(N+\alpha+\beta)}\int_{\mathbb R^{N}}(|\nabla u |^{2}+|\nabla v |^{2}+|\nabla \xi_{n}|^{2}+|\nabla \gamma_{n}|^{2})\\
&\ +\frac{1}{2(N+\alpha+\beta)}\int_{\mathbb R^{N}}\bigl((\alpha+\beta-2)A(x)-\nabla A(x)\cdot x\bigr)(u ^{2}+\xi_{n}^{2})\\
&\ +\frac{\alpha+\beta-2}{2(N+\alpha+\beta)}\int_{\mathbb R^{N}}(Bv ^{2}+u^{2}|\nabla u |^{2}+v ^{2}|\nabla v |^{2}+\gamma_{n}^{2}+\xi_{n}^{2}|\nabla \xi_{n}|^{2}\\
&\ +\gamma_{n}^{2}|\nabla \gamma_{n}|^{2})\\
>&\ \frac{\alpha+\beta}{2(N+\alpha+\beta)}\int_{\mathbb R^{N}}(|\nabla u |^{2}+|\nabla v |^{2}+|\nabla \xi_{n}|^{2}+|\nabla \gamma_{n}|^{2})\\
&\ +\frac{1}{2(N+\alpha+\beta)}\int_{\mathbb R^{N}}\bigl((\alpha+\beta-2)A(x)-\nabla A(x)\cdot x\bigr)(u ^{2}+t_n^{N+2}\xi_{n}^{2})\\
&\ +\frac{\alpha+\beta-2}{2(N+\alpha+\beta)}\int_{\mathbb R^{N}}(Bv ^{2}+u^{2}|\nabla u |^{2}+v ^{2}|\nabla v |^{2}+t_n^{N+2}\gamma_{n}^{2}\\
&\ +t_n^{N+2}\xi_{n}^{2}|\nabla \xi_{n}|^{2}+t_n^{N+2}\gamma_{n}^{2}|\nabla \gamma_{n}|^{2})\endaligned$$
$$\aligned=&\ I\bigl((\xi_{n})_{t_n}, (\gamma_{n})_{t_n}\bigr)+\frac{\alpha+\beta}{2(N+\alpha+\beta)}\int_{\mathbb R^{N}}(|\nabla u|^{2}+|\nabla v|^{2})\\
&\ +\frac{1}{2(N+\alpha+\beta)}\int_{\mathbb R^{N}}\bigl((\alpha+\beta-2)A(x) -\nabla A(x)\cdot x\bigr)u^{2}\\
&\ +\frac{\alpha+\beta-2}{2(N+\alpha+\beta)}\int_{\mathbb R^{N}}(Bv^{2}+u^{2}|\nabla u|^{2}+v^{2}|\nabla v|^{2})\\
\geq&\ m , \endaligned$$
which is also a contradiction.

Therefore, $(u,v)\in \mathcal{M}$.  By using  Lebesgue dominated convergence theorem, Fatou's Lemma, $(A_{2})$  and $(u_{n},v_{n})\in \mathcal{M} $, we may get
$$\aligned m&=I(u,v)-\frac{1}{N+\alpha+\beta}\mathcal{G}(u,v)\\
&=\frac{\alpha+\beta}{2(N+\alpha+\beta)}\int_{\mathbb R^{N}}(|\nabla u |^{2}+|\nabla v |^{2})\\
&\ \ \ \ +\frac{1}{2(N+\alpha+\beta)}\int_{\mathbb R^{N}}\bigl((\alpha+\beta-2)A(x)-\nabla A(x)\cdot x\bigr)u ^{2}\\
&\ \ \ \ +\frac{\alpha+\beta-2}{2(N+\alpha+\beta)}\int_{\mathbb R^{N}}(Bv^{2}+u^{2}|\nabla u |^{2}+v ^{2}|\nabla v |^{2})\\
&\leq \liminf_{n\rightarrow+\infty}\bigg(\frac{\alpha+\beta}{2(N+\alpha+\beta)}\int_{\mathbb R^{N}}(|\nabla u_{n} |^{2}+|\nabla v_{n} |^{2})\\
&\ \ \ \ +\frac{1}{2(N+\alpha+\beta)}\int_{\mathbb R^{N}}\bigl((\alpha+\beta-2)A(x)-\nabla A(x)\cdot x\bigr)u_{n} ^{2}\\
&\ \ \ \ +\frac{\alpha+\beta-2}{2(N+\alpha+\beta)}\int_{\mathbb R^{N}}(Bv_{n}^{2}+u_{n}^{2}|\nabla u_{n} |^{2}+v_{n} ^{2}|\nabla v_{n} |^{2})\bigg)\\
&=\liminf_{n\rightarrow+\infty}\bigg(I( u_{n} ,v_{n} )-\frac{1}{N+\alpha+\beta}\mathcal{G}(u_{n},v_{n})\bigg)\\
&=m,\endaligned$$
which implies that $(u_{n},v_{n})\rightarrow (u,v)$  in $X$ and $I(u,v)=m$.

Having a minimum of  $I|_{\mathcal{M}}$, the fact that it is indeed a solution of  (\ref{eq1.1}), is based on a general idea used in \cite[Lemma 2.5]{lww}.
\vskip3.73pt
\noindent{\bf Proof of Theorem \ref{th1.1}.} Let 
$(\tilde{u},\tilde{v})\in \mathcal{M}$  be a minimizer of the functional 
$I|_{\mathcal{M}}$. We show that $I'(\tilde{u},\tilde{v})=0$. 
By using Lemma \ref{le2.2}, 
$$I(\tilde{u},\tilde{v})=\inf_{(u,v)\in X }\max\limits_{t>0}I(u_{t},v_{t})=m.$$
Now, we argue by contradiction by assuming that $(\tilde{u},\ \tilde{v})$ is not a weak\\

\noindent solution of (\ref{eq1.1}). Then, we can choose $\phi_{1},\ \phi_{2}\in C_{0}^{\infty}(\mathbb R^{N})\cap \mathcal{V}$  such that
$$\aligned &\ \langle I'(\tilde{u},\tilde{v}),(\phi_{1},\phi_{2})\rangle\\
=& \int_{\mathbb R^{N}}\bigg( \nabla \tilde{u}\nabla \phi_{1}+ \nabla \tilde{v}\nabla \phi_{2}+\nabla(\tilde{u}^{2})\nabla(\tilde{u}\phi_{1})+\nabla(\tilde{v}^{2})\nabla(\tilde{v}\phi_{2}) \\
&\ +A(x)\tilde{u}\phi_{1}+B\tilde{v}\phi_{2}-\frac{2\alpha}{\alpha+\beta}|\tilde{u}|^{\alpha-2}\tilde{u}|\tilde{v}|^{\beta}\phi_{1}-\frac{2\beta}{\alpha+\beta}|\tilde{v}|^{\beta-2}\tilde{v}|\tilde{u}|^{\alpha}\phi_{2}\bigg)\\
<&\ -1.\endaligned$$
Then we fix $\varepsilon> 0$ sufficiently small such that
$$\bigl\langle I' (\tilde{u}_{t}+\sigma\phi_{1},\tilde{v}_{t}+\sigma\phi_{2}),(\phi_{1},\phi_{2}) \bigr\rangle\leq -\frac{1}{2},\ \forall\ |t-1|,|\sigma|\leq\varepsilon $$
and introduce a cut-off function $0\leq \zeta\leq1 $ such that $\zeta(t)=1 $ for $|t-1|\leq\frac{\varepsilon}{2}$ and $\zeta(t)=0 $ for $|t-1|\geq\varepsilon. $
For $ t\geq0$, we define
$$\gamma_{1}(t):=\left\{\aligned &\tilde{u}_{t}, &\hbox{if}\ |t-1|\geq\varepsilon,\\
&\tilde{u}_{t}+\varepsilon\zeta(t)\phi_{1}, &\hbox{if}\ |t-1|<\varepsilon,\endaligned\right. $$
$$\gamma_{2}(t):=\left\{\aligned &\tilde{v}_{t}, &\hbox{if}\ |t-1|\geq\varepsilon,\\
&\tilde{v}_{t}+\varepsilon\zeta(t)\phi_{2}, &\hbox{if}\ |t-1|<\varepsilon.\endaligned\right. $$
Note that $ \gamma_{1}(t)$ and $\gamma_{2}(t)$ are continuous curves in the metric space $ (X,d)$  and, eventually choosing a smaller $ \varepsilon$, we get that  for  $|t-1|<\varepsilon$,
$$d_{X}\bigl((\gamma_{1}(t),\gamma_{2}(t)),(0,0)\bigr)>0.$$
\vskip4pt
{\bf Claim: } $\sup\limits_{t\geq0}I(\gamma_{1}(t),\gamma_{2}(t))<m.$

Indeed, if  $|t-1|\geq\varepsilon $, then $I(\gamma_{1}(t),\gamma_{2}(t))=I(\tilde{u}_{t},\tilde{v}_{t})<I(u,v)=m$. If $ |t-1|<\varepsilon$, by using the mean value theorem to the $C^{1}$  map $[0,\varepsilon]\ni \sigma\mapsto I(\tilde{u}_{t}+\sigma\zeta(t)\phi_{1},\tilde{v}_{t}+\sigma\zeta(t)\phi_{2})\in\mathbb R$, we find, for a suitable $\bar{\sigma}\in(0,\varepsilon)$,
$$\aligned &\ I(\tilde{u}_{t}+\sigma\zeta(t)\phi_{1},\tilde{v}_{t}+\sigma\zeta(t)\phi_{2})\\
=&\ I(\tilde{u}_{t},\tilde{v}_{t})+\langle I'(\tilde{u}_{t}+\bar{\sigma}\zeta(t)\phi_{1},\tilde{v}_{t}+\bar{\sigma}\zeta(t)\phi_{2}),(\zeta(t)\phi_{1},\zeta(t)\phi_{2})\rangle\\
\leq&\ I(\tilde{u}_{t},\tilde{v}_{t})-\frac{1}{2}\zeta(t)\\
<&\ m.\endaligned$$
To conclude, we observe that $\mathcal{G}(\gamma_{1}(1-\varepsilon),\gamma_{2}(1-\varepsilon))>0$ and $\mathcal{G}(\gamma_{1}(1+\varepsilon),\gamma_{2}(1+\varepsilon))<0$.  By the continuity of the map $t\mapsto \mathcal{G}(\gamma_{1}(t),\gamma_{2}(t))$  there exists $t_{0}\in(1-\varepsilon,1+\varepsilon) $  such that $\mathcal{G}(\gamma_{1}(t_{0}),\gamma_{2}(t_{0}))=0$. Namely,
$$ (\gamma_{1}(t_{0}),\gamma_{2}(t_{0}))=(\tilde{u}_{t_{0}}+\varepsilon\zeta(t_{0})\phi_{1},\tilde{v}_{t_{0}}+\varepsilon\zeta(t_{0})\phi_{2})\in \mathcal{M}$$  and $I(\gamma_{1}(t_{0}),\gamma_{2}(t_{0}))<m$, this is a contradiction.

In addition, from the definition of $\mathcal{V}$ and the fact that 
$\det(\eta)=-1, (u(\eta x),v(\eta x))
=(\det(\eta)u(x),\det(\eta)v(x))=(-u(x),-v(x)).$ So $(u,v)$ will 
change sign when $(y_{ 1 },y_{ 2})$ cross perpendicularly the half lines 
$y_{ 2}=\pm y_{ 1} \hbox{tan}\bigl(\frac{\pi}{j}\bigr)(y_{1}\geq0), j = 1,2,\ldots,s.$ Hence $(u,v)$ is a nodal solution with at least $2s$ nodal domains.
\end{proof}
\vskip10pt
\noindent{\bf Acknowledgments}
\vskip16pt
The authors would like to gratefully acknowledge support by National Natural Science Foundation of China 
(No. 11871152)  and Key Project of Natural Science Foundation of Fujian (No. 2020J02035).
We thank wish to thank the anonymous 
referee vrery much for the careful reading and valuable comments.

\end{document}